\documentclass[12pt]{article}

\usepackage{amssymb}
\usepackage{booktabs}
\usepackage{multirow}
\usepackage{amsmath}
\usepackage{amsthm}
\usepackage{amsfonts}

\newcommand{\mmp}{\mathbb{P}}

\newcommand{\od}{\overset{d}{=}}
\newcommand{\dod}{\overset{d}{\to}}

\newcommand{\me}{\mathbb{E}}
\newcommand{\mr}{\mathbb{R}}
\newcommand{\mn}{\mathbb{N}}

\newcommand{\be}{{\rm beta}}
\newcommand{\Be}{{\rm B}}

\newcommand{\lin}{\underset{n\to\infty}{\lim}}

\newtheorem{thm}{Theorem}[section]
\newtheorem{lemma}[thm]{Lemma}

\newtheorem{cor}[thm]{Corollary}

\newtheorem{assertion}[thm]{Proposition}

\theoremstyle{definition}

\theoremstyle{remark}
\newtheorem{rem}[thm]{Remark}
\begin{document}
\title{On asymptotics of the beta-coalescents}
\date{\today}

\author{
Alexander Gnedin\thanks{
Queen Mary, University of London,
e-mail: a.gnedin@qmul.ac.uk},
\quad
Alexander Iksanov\thanks{
National T. Shevchenko University of Kiev,
e-mail: iksan@univ.kiev.ua},\\
Alexander
Marynych\thanks{
Technical University of Eindhoven,
e-mail: O.Marynych@tue.nl}
\quad and\quad
Martin M\"ohle\thanks{
University of T\"ubingen, e-mail: martin.moehle@uni-tuebingen.de} }
\maketitle

\begin{abstract}
\noindent We show that the total number of collisions in the
exchangeable coalescent process driven by the $\be\, (1,b)$
measure converges in distribution to a 1-stable law, as the
initial number of particles goes to infinity. The stable limit law
is also shown for the total branch length of the coalescent tree.
These results were known previously for the instance $b=1$, which
corresponds to the Bolthausen--Sznitman coalescent. The approach we
take is based on estimating the quality of a renewal approximation
to the coalescent in terms of a suitable Wasserstein distance.
Application of the method to $\be\,(a,b)$-coalescents with $0<a<1$
leads to a simplified derivation of the known $(2-a)$-stable
limit. We furthermore derive asymptotic expansions for the
moments of the number of collisions and of
the total branch length for the $\be\,(1,b)$-coalescent by
exploiting the method of sequential approximations.


\vspace{5mm}


\vspace{5mm}


\end{abstract}

\section{Introduction} \label{intro}

Pitman \cite{Pit} and Sagitov \cite{Sag} introduced exchangeable
coalescent processes with multiple collisions, also known as
$\Lambda$-coalescents. A counting process associated with the
$\Lambda$-coalescent is a Markov chain
$\Pi_n=\big(\Pi_n(t)\big)_{t\geq 0}$ with right-continuous paths,
which starts with $n$ particles $\Pi_n(0)=n$ and terminates when a
sole particle remains. The particles merge according to the rule:
 for each $t\geq 0$
when the number of particles is $\Pi_n(t)=m>1$, each $k$ tuple of them is
merging in one particle at probability rate
\begin{equation}\label{lrates}
\lambda_{m,\,k}=\int_{0}^1 x^{k}(1-x)^{m-k}x^{-2}\Lambda({\rm
d}x), \qquad 2\leq k\leq m,
\end{equation}
where $\Lambda$ is a given finite measure on the unit interval.
The event of merging of two or more particles is called {\it collision}.
By every collision $\Pi_n$ jumps to a smaller value.
When $\Lambda$ is a Dirac mass at $0$ the $\Lambda$-coalescent
is the classical Kingman
coalescent \cite{Kingman}, in which every pair of particles is merging at the  unit rate and only binary mergers
are possible.
Another eminent instance, known as
the Bolthausen-Sznitman coalescent \cite{bolt_sznit}, appears
when $\Lambda$ is the Lebesgue measure on $[0,1]$.

The subclass of {\it beta-coalescents} are the processes driven by some
beta measure on $[0,1]$ with density
\begin{equation}\label{beta}
\Lambda({\rm d}x)/{\rm d}x= \frac{1}{\Be(a,b)}
x^{a-1}(1-x)^{b-1},\qquad a,b>0,
\end{equation}
where $\Be(\cdot,\cdot)$ denotes Euler's beta function. This
class is amenable to analysis due to the fact that the transition
rates (\ref{lrates}) can be expressed in terms of $\Be(\cdot,\cdot)$.
For this reason and due to multiple connections with
L{\'e}vy processes and random trees, the beta coalescents were the
subject of intensive research \cite{BeGa, BePi, birkner,
 DIMR1, DIMR2, gold_martin, IksMoe, Pit}. We refer to \cite{BerestLect}
for a survey and further references.

In this paper we study beta-coalescents with parameter $0< a\leq
1$. Specifically, we are interested in the total number of
collisions $X_n$ and the total branch length of the coalescent
tree $L_n$. Note that $X_n$ is equal to the total number of
particles born by collisions, and $L_n$ is the cumulative lifetime
of all particles from the start of the process to its termination.
The variable $L_n$ is closely related to the number of segregating
sites $M_n$, the connection being that given $L_n$ the
distribution of $M_n$ is Poisson with mean $rL_n$ for some fixed
mutation rate $r>0$.

Our first main new contribution is the proof of a $1$-stable limit law
for $X_n$ and $L_n$ as $n\to\infty$. As in much of the previous
work (see, for instance, \cite{GneYak} and \cite{IksMoe2}) we use
a renewal approximation to $\Pi_n$. A novel element in this
context is estimating the quality of approximation in terms of a
Wasserstein distance.

The second new contribution are asymptotic expansions for the
moments of $X_n$, $L_n$ and $M_n$ for the $\be\,(1,b)$-coalescent
with arbitrary parameter $b>0$. These expansions are obtained
independently from the weak limiting results mentioned before. The proofs
are based on the method of sequential approximations similar to those
used in \cite{IksMoeMar}.

The rest of the paper is organized as follows. Section
\ref{survey} gives a summary of some results on limit laws related
to the beta-coalescents. In Section \ref{chain} general properties
of the block-counting Markov chain and basic recurrences are discussed
and the main results are stated. Section \ref{distances_section}
recalls the definition and properties of a Wasserstein distance.
In Section \ref{proofs_section} we provide proofs of the main results.
Some auxiliary lemmas are collected in the appendix.

\section{A summary of limit laws for beta-coalescents}\label{survey}

The tables in this section summarize the limit laws for $X_n$,
$L_n$ and the absorption time of the
 coalescent
$\tau_n:=\min\{t:\Pi_n(t)=1\}$.
The distributions which appear in the tables will be denoted as
follows

\begin{itemize}
\item[(i)] $\cal N$, standard normal,

\item[(ii)] ${\cal S}_\alpha$ with $1<\alpha<2$, (spectrally negative) $\alpha$-stable  distribution with characteristic function

\begin{equation}\label{cf_sta}
z\mapsto \exp\big\{|z|^{\alpha}\big(\cos (\pi \alpha /2)+{\rm
i}\sin (\pi\alpha/2){\rm sgn}(z)\big)\big\}, \ z\in\mr,
\end{equation}
\item[(iii)]  $\mathcal{S}_1$ (spectrally negative) 1-stable distribution with characteristic function
\begin{equation}\label{cf_st1}
z\mapsto \exp\big\{-|z|(\pi/2-{\rm i}\log|z|\,{\rm sgn}(z))\big\},
\ z\in\mr,
\end{equation}

\item[(iv)] ${\cal E}_\gamma (a,b)$ with $a,b,\gamma>0,$ distribution of the {\it exponential functional}
$\int_0^\infty \exp(-\gamma S_{a,\,b}(t)){\rm d}t$, where
$\big(S_{a,\,b}(t)\big)_{t\geq 0}$ is a drift-free subordinator
with the Laplace exponent
$$\Phi_{a,\,b}(z)=\int_0^1\big(1-(1-x)^z\big)x^{a-3}(1-x)^{b-1}{\rm d}x,~~~~~z\geq 0,$$
\item[(v)] $\cal G$, Gumbel with distribution function $x\mapsto \exp\big(-e^{-x}\big), x\in {\mathbb R}$,
\item[(vi)] $\rho$,  convolution of infinitely many
 exponential laws with rates $i(i-1)/2$, $i\geq 2$.

\end{itemize}

\vskip0.3cm

\begin{center}{\bf Table 1:} Limit distributions of $(X_n-a_n)/b_n$ for $\be\,(a,b)$-coalescents.\\
\begin{small}
\begin{tabular}{cccccc}
\toprule
$a$ & $b$ & $a_n$ & $b_n$ & distribution & source\\
\midrule
\multirow{2}*{$0<a<1$} & \multirow{2}*{$b>0$} & \multirow{2}*{$(1-a)n$} & \multirow{2}*{$(1-a)
n^{1/(2-a)}$} & \multirow{2}*{${\cal S}_{2-a}$} & \cite{IksMoe2}($b=1$), \cite{GneYak},\\
& & & & &this paper\\
\midrule
\multirow{2}*{$a=1$} & \multirow{2}*{$b>0$} & $n(\log n)^{-1} + $ & \multirow{2}*{$\frac{n}{(\log n)^{2}}$}
& \multirow{2}*{${\cal S}_1$} & \cite{DIMR2,IksMoe}($b=1$),\\
& & \hspace{1mm} $n\log\log n (\log n)^{-2}$& & &this paper\\
\midrule $1<a<2$ & $b>0$ & 0 & $\frac{\Gamma(a)}{2-a} \, n^{2-a}$ &
{${\cal E}_{2-a}(a,b)$}
& \cite{GIM_coal,Haas}\\
\midrule
$a=2$ & $b>0$ & $(2r_1)^{-1}(\log n)^2$ & $(3^{-1}r_1^{-3}r_2\log^3 n)^{1/2}$ & {$\cal N$} & \cite{GIM_coal,IksMoeMar}\\
\midrule
$a>2$ & $b>0$ & $m_1^{-1}\log n$ & $(m_1^{-3}m_2\log n)^{1/2}$ & {$\cal N$} & \cite{GIM_coal,GneIksMoe}\\
\bottomrule
\end{tabular}
\end{small}
\end{center}

\paragraph{Notation and comments:}
$r_1=\zeta(2,b), \ \ r_2=2\zeta(3,b),$ where $\zeta(\cdot,\cdot)$
is the Hurwitz zeta function; $m_1=\Psi(a-2+b)-\Psi(b), \ \ m_2=\Psi'(b)-\Psi'(a-2+b)$,
where $\Psi(\cdot)$ is the logarithmic derivative of the gamma function.

For the Bolt\-hausen-Sznit\-man coalescent the limit law of $X_n$
was first obtained in \cite{DIMR2} using singularity analysis of
generating functions. A probabilistic proof of this result
appeared in \cite{IksMoe}, where  a coupling with a random walk
with barrier was exploited, and the technique was further extended
in \cite{IksMoe2} to study collisions in the
$\be\,(a,1)$-coalescents with $a\in (0,2)$. 
The aforementioned limit laws for $a>1$ are specializations of
results for more general $\Lambda$-coalescents with {\it dust
component}, i.e., those driven by the measures $\Lambda$ such that
$\int_0^1x^{-1}\Lambda({\rm d}x)<\infty$
\cite{GIM_coal,GneIksMoe,GneYak,Haas}. For  Kingman's
coalescent we have $X_n=n-1$ for all $n\in\mn$.

In the next two tables the value $a=0$ corresponds to
Kingman's coalescent. \vskip0.3cm
\begin{center}{\bf Table 2:} Limit distributions of $(\tau_n-a_n)/b_n$ for $\be\,(a,b)$-coalescents.\\
\begin{small}
\begin{tabular}{cccccc}
\toprule
$a$ & $b$ & $a_n$ & $b_n$ & distribution & source\\
\midrule
$a=0$ &  & $0$ & $1$ & $\rho$ & \cite{Tav}\\
\midrule
$a=1$ & $b=1$ & $\log\log n$ & $1$ & ${\cal G}$ & \cite{gold_martin, FreMoe}\\
\midrule
$1<a<2$ & $b>0$ & ${\tt m}^{-1}\log n$ & $({\tt m}^{-3}{\tt s}^2\log n)^{1/2}$ & ${\cal N}$ & \cite{GIM_coal}\\
\midrule
$a=2$ & $b>0$ & $c_1^{-1}\log n$ & $(c_1^{-3}c_2 \log n)^{1/2}$ & ${\cal N}$ & \cite{GIM_coal}\\
\midrule
$a>2$ & $b>0$ & $(\gamma m_1)^{-1}\log n$ &
$\gamma^{-1}(m_1^{-3}(m_2+m_1^2)\log n)^{1/2}$ & ${\cal N}$ & \cite{GIM_coal,GneIksMoe}\\
\bottomrule
\end{tabular}
\end{small}
\end{center}
\textbf{Notation and comments:}
The constants ${\tt m}$ and ${\tt s}^2$ are
\begin{eqnarray*}
{\tt m}&=&\frac{a+b-1}{(a-1)(2-a)}\bigg(1-(a+b-2)\big(\Psi(a+b-1)-\Psi(b)\big)\bigg),\\
{\tt s}^2&=& \frac{a+b-1}{(a-1)(2-a)}\times
\bigg(2\big(\Psi(a+b-1)-\Psi(b)\big)\nonumber\\
&-&(a+b-2)\big((\Psi(a+b-1)-\Psi(b))^2+\Psi^\prime(b)-\Psi^\prime(a+b-1)\big)\bigg),
\end{eqnarray*}
$c_1=b(b+1)\zeta(2,b), \ \ c_2=2b(b+1)\zeta(3,b).$
The constants $m_1$ and $m_2$ are the same as in Table 1, and for $a>2$
$$\gamma=\frac{(a-1+b)(a-2+b)}{(a-1)(a-2)}.$$

In the case $a\in (0,1)$, $b>0$ the $\be\,(a,b)$-coalescent has the
property of coming down from infinity \cite{Sch}, which implies
that $\tau_n$ weakly converges without any normalization to some
limiting law, which is not known explicitly. The result for $a>1$
is a special case of Theorem 4.3 in \cite{GIM_coal}. The case
$a=1$ and $b\neq 1$ is open; in this case the coalescent does not
come down from infinity.

\vskip0.3cm

\begin{center}{\bf Table 3:} Limit distributions of  $(L_n-a_n)/b_n$ for $\be\,(a,b)$-coalescents.\\
\begin{small}
\begin{tabular}{cccccc}
\toprule
$a$ & $b$ & $a_n$ & $b_n$ & distribution & source\\
\midrule
$a=0$ &  & $2\log n$ & $2$ & $\cal G$ & \cite{DIMR1, Tav}\\
\midrule
$0<a<\frac{3-\sqrt{5}}{2}$ & $b=2-a$ & $c_1n^{a}$ & 1 & exists & \cite{Kersting}\\
\midrule
$a=\frac{3-\sqrt{5}}{2}$ & $b=2-a$ & $c_1n^{a}$ & $c_2(\log n)^{\alpha^{-1}}$ & ${\cal S}_{2-a}$ &\cite{Kersting}\\
\midrule
$\frac{3-\sqrt{5}}{2}<a<1$ & $b=2-a$ & $c_1n^{a}$ & $c_2(\beta n^{-\beta})^{\alpha^{-1}}$ & ${\cal S}_{2-a}$
&\cite{Kersting}\\
\midrule
\multirow{2}*{$a=1$} & \multirow{2}*{$b>0$} & $n(b\log n)^{-1} +$  & \multirow{2}*{$\frac{n}{b(\log n)^{2}}$}
 & \multirow{2}* {${\cal S}_1$}  & \cite{DIMR1}($b=1$),\\
& & \hspace{1mm} $b^{-1}n\log\log n (\log n)^{-2}$& & &this paper\\
\midrule
$a>1$ & $b>0$ & 0 & $\Be(a,b)n$ & ${\cal E}_1(a,b)$ & \cite{Moe06, Moehle2009}\\
\bottomrule
\end{tabular}
\end{small}
\end{center}

\noindent
\textbf{Notation and comments:} The constants are $\alpha=2-a$, $\beta=1+\alpha-\alpha^2$,
$c_1=\frac{\Gamma(\alpha+1)(\alpha-1)}{2-\alpha}$,
$c_2=\frac{\Gamma(\alpha+1)(\alpha-1)^{1+\alpha^{-1}}}{\cos{(\pi\alpha/2)}\Gamma^{\alpha^{-1}}(2-\alpha)}$.

In \cite{Moe06} the weak convergence of properly normalized $L_n$
was proved for $\Lambda$-coalescents with dust component. In
particular, that result covered the $\be\,(a,b)$-coalescents with
$a>1$. Although some partial results for $a\in (0,1)$ and $b>0$
were obtained in \cite{Delmas}, this case with $b\neq 2-a$ remains
open.

\section{Main results}\label{chain}

For the general $\Lambda$-coalescent, the Markov chain
$\Pi_n$ is a pure-death process which jumps from state
$m$ to $m-k+1$ at rate $\binom{m}{k}\lambda_{m,\,k}$, where
$\lambda_{m,\,k}, 2\leq k\leq m,$ is given by \eqref{lrates}. The
total transition rate from state $m\geq 2$ is
\begin{equation}\label{total_rates}
\lambda_m:=\sum_{k=2}^m
\binom{m}{k}
\lambda_{m,\,k}=\int_0^1
\big(1-mx(1-x)^{m-1}-(1-x)^m\big)x^{-2}\Lambda({\rm d}x).
\end{equation}
The first decrement $I_n$ of $\Pi_n$ has distribution
$${\mathbb P}\{I_n=k\}=
\binom{n}{k+1}
\frac{\lambda_{n,\,k+1}}{\lambda_n},\;\;1\leq k\leq n-1.$$
The strong Markov property of the coalescent entails the
distributional recurrences
\begin{equation}\label{xn_rec}
X_1  = 0,\;\;X_n\od 1+X'_{n-I_n}, \ \ n\in\mn\backslash\{1\};
\end{equation}
\begin{equation}\label{tn_rec}
\tau_1 = 0,\;\;\tau_n\od T_n+\tau'_{n-I_n}, \ \
n\in\mn\backslash\{1\};
\end{equation}
\begin{equation}\label{ln_rec}
L_1 = 0,\;\;L_n\od nT_n+L'_{n-I_n}, \ \ n\in\mn\backslash\{1\},
\end{equation}
where $T_n$ denotes the time of the first collision, hence $T_n$
has the exponential law with parameter $\lambda_n$; $X_k^\prime$
(respectively, $\tau_k^\prime$, $L_k^\prime$) is independent of
$I_n$ (respectively, $(T_n, I_n)$) and is distributed like $X_k$
(respectively, $\tau_k$, $L_k$), for each $k\in\mn$.

Letting $\Lambda$ be defined by \eqref{beta} with $a\in (0,1]$ denote by
\begin{equation}\label{pnka}
p^{(a)}_{n,k}:=\mmp\{I_n=n-k\},\;\;k=1,\ldots, n-1.
\end{equation}
Using the leading terms of asymptotic relations
\eqref{total_rates_b}, \eqref{total_rates_a} and
\eqref{total_rates_1} we infer
$$
\lim_{n\to\infty}p^{(a)}_{n,n-k}=\frac{(2-a)\Gamma(k+a-1)}{\Gamma(a)(k+1)!}=:p^{(a)}_{k},\;\;k\in\mn,
$$
hence
\begin{equation}\label{ind_conv}
I_n\dod \xi,\;\;n\to\infty,
\end{equation}
where $\xi$ is a random variable with distribution $(p^{(a)}_k)_{k\in\mn}$.

Consider a zero-delayed random walk $\big(S_n\big)_{n\in\mn_0}$
defined by $$S_0:=0, \ \ S_n:=\xi_1+\ldots+\xi_n, \ \ n\in\mn,$$
where $\big(\xi_j\big)$ are independent copies of $\xi$ with
distribution $(p^{(a)}_k)_{k\in\mn}$, and let
$\big(N_n\big)_{n\in\mn_0}$ be the associated first-passage time
sequence defined by
$$
N_n=\inf\{k\geq 0:S_k \geq n\},\;\;n\in\mn_0.
$$
It is plain that
\begin{equation}\label{renewal_rec}
N_0=0,\;\;N_n\od 1+N'_{n-\xi\wedge
n}=1+N^\prime_{n-\xi}1_{\{\xi<n\}},\;\;n\in\mn,
\end{equation}
where $N'_k$ is independent of $\xi$ and distributed like $N_k$,
for each $k\in\mn$. Comparing \eqref{xn_rec} and
\eqref{renewal_rec} one can expect that if $N_n$ (properly
centered and normalized) converges weakly to some proper and
non degenerate probability law then the same is true for $X_n$
(with the same centering and normalization). This is what we mean
by a renewal approximation mentioned in the Introduction.
This idea was exploited in \cite{GneYak} (for $a\in (0,1)$, $b>0$)
and in  \cite{IksMoe2} (for $a\in (0,1]$, $b=1$) to derive the
limit distribution of $X_n$ from that of $N_n$.
We shall use a method based on probability metrics to show the
stable limits for $a\in (0,1]$ and $b>0$. 

\begin{thm}\label{main_thm} As $n\to\infty$
the number of collisions $X_n$ in the $\be\,(a,b)$-coalescent
satisfies
\begin{itemize}
\item[\rm(i)]
for $0<a<1$ and $b>0$
$$
\frac{X_n-(1-a)n}{(1-a)n^{1/(2-a)}}\dod \mathcal{S}_{2-a},
$$
\item[\rm(ii)]
for $a=1$ and $b>0$,
$$
\frac{\log^2 n}{n} X_n-\log n-\log\log n \dod \mathcal{S}_1.
$$
\end{itemize}
\end{thm}

As a consequence of our main theorem we also obtain a weak limit
for the total branch length $L_n$ and the number of segregating
sites $M_n$ (see \cite{Moe06}) of the $\be\,(1,b)$-coalescent.
\begin{cor}\label{tbl_cor}
For the total branch length $L_n$ in the $\be\,(1,b)$-coalescent
we have as $n\to\infty$
$$
\frac{b\log^2 n}{n}L_n-\log n-\log\log n \ \dod \ \mathcal{S}_1.
$$
\end{cor}
\begin{cor}\label{seg_cites_cor}
For the number of segregating sites $M_n$ in the
$\be\,(1,b)$-coalescent we have as $n\to\infty$
$$
\frac{b\log^2 n}{r n}M_n-\log n-\log\log n \ \dod \
\mathcal{S}_1,
$$
where $r>0$ is the rate of the homogeneous Poisson process on
branches of the coalescent tree.
\end{cor}
   We now turn to the moments of $X_n$, $L_n$ and $M_n$.
   An analysis of these moments provides further insight into the
   structure of these functionals. Our next result concerns the
   asymptotics of the moments of the number of collisions $X_n$ in
   the $\be\,(1,b)$-coalescent.
\begin{thm} 
   \label{main1}
   Fix $b\in (0,\infty)$ and $j\in\mn_0$. The $j$th moment of the number
   of collisions $X_n$ in the $\be\,(1,b)$-coalescent has the
   asymptotic expansion
   \begin{equation} \label{xnmomentasy}
      \me X_n^j\ =\ \frac{n^j}{\log ^jn}
      \bigg(
         1+\frac{m_j}{\log n} + O\bigg(\frac{1}{\log^2n}\bigg)
      \bigg),
      \qquad n\to\infty,
   \end{equation}
   where the sequence $(m_j)_{j\in\mn_0}$ is recursively defined via
   $m_0:=0$ and $m_j:=m_{j-1}+\kappa_j/j$ for $j\in\mn$, with
   $\kappa_j:=(j+b-1)\Psi(j+b)+j-(b-1)\Psi(b)$, $j\in\mn_0$.
\end{thm}
   For some more information on the coefficients $m_j$, $j\in\mn$,
   we refer the reader to Eq.~(\ref{mj}) in the proof of the
   following Corollary \ref{centralmoments}, which provides
   asymptotic expansions for the central moments
   of $X_n$ in the $\be\,(1,b)$-coalescent.
\begin{cor} 
   \label{centralmoments}
   Fix $b\in (0,\infty)$ and $j\in\mn\setminus\{1\}$. The $j$th central
   moment of the number of collisions $X_n$ in the $\be\,(1,b)$-coalescent
   has the asymptotic expansion
   \begin{equation} \label{xncentralmomentasy}
      \me(X_n-\me X_n)^j
      \ =\ \frac{(-1)^j}{j}\Be(b,j-1)\frac{n^j}{\log^{j+1}n}
      + O\bigg(\frac{n^j}{\log^{j+2}n}\bigg),\qquad n\to\infty.
   \end{equation}
   In particular, ${\rm Var}(X_n)=(2b)^{-1}n^2/\log^3n+O(n^2/\log^4n)$ as $n\to\infty$.
\end{cor}
\begin{rem}
   For $b=1$, Eq.~(\ref{xncentralmomentasy}) reduces to the
   asymptotic expansion (see Panholzer \cite[p.~277 or Theorem 2.1.
   with $\alpha=0$]{Pan2})
   $$
   \me(X_n-\me X_n)^j
   \ =\ \frac{(-1)^j}{j(j-1)}\frac{n^j}{\log^{j+1}n} + O\bigg(\frac{n^j}{\log^{j+2}n}\bigg),
   \qquad n\to\infty
   $$
   of the $j$th central moment of the number
   of collisions $X_n$ for the Bolthausen--Sznitman $n$-coalescent.
\end{rem}
The last result concerns the moments end central moments of the total
branch length $L_n$
of the $\be\,(1,b)$-coalescent. 
\begin{assertion} 
\label{main3}
   Fix $b\in (0,\infty)$ and $j\in\mn_0$. The $j$th moment of the total
   branch length $L_n$ of the $\be\,(1,b)$-coalescent has the asymptotic
   expansion
   \begin{equation} \label{lnmomentasy}
      \me L_n^j\ =\
      \frac{1}{b^j}\frac{n^j}{\log^jn}
      \bigg(
         1+\frac{m_j}{\log n}
         +O\bigg(\frac{1}{\log^2n}\bigg)
      \bigg),
      \qquad n\to\infty,
   \end{equation}
   where the sequence $(m_j)_{j\in\mn_0}$ is defined as in Theorem
   \ref{main1}. Moreover, for $j\in\{2,3,\ldots\}$, the $j$th central
   moment of $L_n$ has the asymptotic expansion
   \begin{equation} \label{lncentralmomentasy}
      \me(L_n-\me L_n)^j\ =\ \frac{(-1)^j}{jb^j}\Be(b,j-1)
      \frac{n^j}{\log^{j+1}n} + O\bigg(\frac{n^j}{\log^{j+2}n}\bigg),
      \qquad n\to\infty.
   \end{equation}
   In particular, ${\rm Var}(L_n)=(2b^3)^{-1}n^2/\log^3n+O(n^2/\log^4n)$ as
   $n\to\infty$.
\end{assertion}
Proposition \ref{main3} indicates that $bL_n$ essentially behaves
like $X_n$, in agreement when comparing Theorem 3.1 (ii) with
Corollary 3.2. The proof of Proposition \ref{main3} works
essentially the same as the analogous proofs of Theorem
\ref{main1} and Corollary \ref{centralmoments} for $X_n$. Instead
of the distributional recurrence (\ref{xn_rec}) for
$(X_n)_{n\in\mn}$ one has to work with the distributional
recurrence (\ref{ln_rec}) for $(L_n)_{n\in\mn}$. Since the
expansion of $\me T_n=1/\lambda_n$ is known (see
Lemma \ref{totalrates}), the proofs concerning $X_n$ are readily
adapted for $L_n$. A proof of Proposition \ref{main3} is therefore
omitted. We finally mention that, for the $\be\,(1,b)$-coalescent
with mutation rate $r>0$, expansions for the moments and central moments
of the number of segregating sites $M_n$ can be easily obtained,
since (see, for example, \cite[p.~1417]{DIMR1}) the descending
factorial moments of $M_n$ are related to the moments of $L_n$ via
$\me(M_n)_j=r^j\me L_n^j$, $j\in\mn_0$.

\newcommand{\QUpsilon}{q}

\section{Probability distances $\chi_T$ and $d_{\QUpsilon}$}\label{distances_section}
For real-valued random variables $X$ and $Y$ and $T>0$ the
$\chi_T$-distance between $X$ and $Y$ is defined by
\begin{equation}
 \chi_T(X,Y)=\sup_{|t|\leq T}\big|\me e^{itX}-\me e^{itY}\big|.
\end{equation}
By the continuity theorem for the characteristic functions
convergence in distribution $Z_n\stackrel{d}{\to}Z$ holds if and
only if $\lin \chi_T(Z_n,Z)=0$, for every $T>0$.  

Let $\mathcal{D}_{\QUpsilon}, \QUpsilon\in (0,1],$ be the set of
probability laws on ${\mathbb R}$ with finite $q$th absolute
moment. Recall that $|x-y|^q$ is a metric on $\mathbb R$. The
associated Wasserstein distance on  $\mathcal{D}_{\QUpsilon}$ is
defined by
\begin{equation}\label{was_dis_rv}
d_{\QUpsilon}(X,Y)=\inf\me|\widehat{X}-\widehat{Y}|^{\QUpsilon},
\end{equation}
where the infimum is taken over all couplings $(\widehat{X},\widehat{Y})$ such that
$X\od \widehat{X}$ and $Y\od\widehat{Y}$.

For ease of reference we  summarize properties of $d_{\QUpsilon}$ in the following proposition.

\begin{assertion}\label{was_dis_prop}
Let $X,Y$ be  random variables with finite $q$th absolute moments.
The Wasserstein distance $d_{\QUpsilon}$ has the following
properties:
 \begin{itemize}
\item[\rm (Dist)] $d_{\QUpsilon}(X,Y)$ only depends on marginal distributions of $X$ and $Y$,
  \item[\rm(Inf)] the infimum in {\rm (\ref{was_dis_rv})}
  is attained for some coupling,
  \item[\rm(Rep)] the {\it Kantorovich-Rubinstein} representation holds
        $$
            d_{\QUpsilon}(X,Y)=\sup_{f\in\mathcal{F}_{\QUpsilon}}|\me f(X) - \me f(Y)|,
        $$
           where $\mathcal{F}_{\QUpsilon}:=\{f\in C(\mr):|f(x)-f(y)|\leq |x-y|^{\QUpsilon},\;\;x,y\in\mr\}$,
  \item[\rm(Hom)] $d_{\QUpsilon}(cX,cY)=|c|^{\QUpsilon}d(X,Y)$ for  $c\in\mr$,
  \item[\rm(Reg)] for $X,Y,Z$ defined on the same probability space
 $d_{\QUpsilon}(X+Z,Y+Z)\leq d_{\QUpsilon}(X,Y)$ provided $Z\in{\mathcal D}_q$ is independent of
$(X,Y)$,
  \item[\rm(Aff)] $d_{\QUpsilon}(X+a,Y+a)=d_{\QUpsilon}(X,Y)$ for  $a\in\mr$,
  \item[\rm(Conv)] for $X, X_n\in\mathcal{D}_{\QUpsilon}$
convergence $d_{\QUpsilon}(X_n,X)\to 0,\;\;n\to\infty$ implies
$X_n\dod X$ and $\me|X_n|^{\QUpsilon}\to\me|X|^{\QUpsilon}$.
 \end{itemize}

\end{assertion}
\begin{proof}
We refer to \cite{GivensShortt,JohnsonSamworth} for most of these
facts. To prove (Reg) choose an independent of $Z$ coupling
$(X',Y')$ on which the infimum in the definition of $d_\QUpsilon$
is
attained. 
Then $X+Z\od X'+Z$, $Y+Z\od Y'+Z$ and the definition of
$d_{\QUpsilon}$ entails
\begin{eqnarray*}
d_{\QUpsilon}(X+Z,Y+Z)&\leq&\me|(X'+Z)-(Y'+Z)|^{\QUpsilon}=\me|X'-Y'|^{\QUpsilon}=d_{\QUpsilon}(X,Y).
\end{eqnarray*}
Property (Conv): the convergence of moments is easy; the rest is a
consequence of Lemma \ref{lemma_distances} to follow.
\end{proof}




\begin{lemma}\label{lemma_distances}
For  $T>0$ and $\QUpsilon\in (0,1]$ there exists constant
$C=C_{T,\QUpsilon}>0$ such that
$$
\sup_{|t|\leq T}|\me e^{itX}-\me e^{itY}|\leq C
d_{\QUpsilon}(X,Y),\;\;n\in\mn.
$$
\end{lemma}
\begin{proof}
Assume that the infimum in the definition of $d_{\QUpsilon}(X,Y)$
is attained on $(\hat{X},\hat{Y})$. It is easy to check that for
arbitrary $\QUpsilon\in(0,1]$
\begin{equation}\label{eit}
 |e^{ix}-e^{iy}|=2\Big|\sin{\frac{x-y}{2}}\Big|\leq 2^{1-\QUpsilon}M_{\QUpsilon}|x-y|^{\QUpsilon},\;\;x,y\in\mr,
\end{equation}
where $M_{\QUpsilon}:=\sup_{u>0}|\sin{u}|u^{-\QUpsilon}<\infty$.
Hence
\begin{eqnarray*}
\sup_{|t|\leq T}|\me e^{itX}-\me e^{itY}|&=&\sup_{|t|\leq T}|\me e^{it\hat{X}}-\me e^{it\hat{Y}}|\leq \sup_{|t|\leq T}
\me |e^{it\hat{X}}-\me e^{it\hat{Y}}|\\
&\overset{\eqref{eit}}{\leq}&
2^{1-\QUpsilon}M_{\QUpsilon}\sup_{|t|\leq T}|t|^{\QUpsilon}\me|\hat{X}-\hat{Y}|^{\QUpsilon}\leq
2^{1-\QUpsilon}M_{\QUpsilon}T^{\QUpsilon}d_{\QUpsilon}(X,Y),
\end{eqnarray*}
as wanted.
\end{proof}

\section{Proofs}\label{proofs_section}
\subsection{Proof of Theorem \ref{main_thm}}
Suppose $a=1$. It is enough to show that
$$
\lin \chi_T\Big(\frac{\log^2 n}{n} X_n-\log n-\log\log n,
\mathcal{S}_1\Big)=0,$$ for every $T>0$.

Using the triangle inequality yields
\begin{eqnarray*}
&&\hspace{-10mm}\chi_T\Big(\frac{\log^2 n}{n} X_n-\log n-\log\log n, \mathcal{S}_1\Big)\leq\\
&&\hspace{-10mm}\chi_T\Big(\frac{\log^2 n}{n} X_n-\log n-\log\log n, \frac{\log^2 n}{n} N_n-\log n-\log\log n \Big)\\
&+&\chi_T\Big(\frac{\log^2 n}{n} N_n-\log n-\log\log n,
\mathcal{S}_1\Big).
\end{eqnarray*}
The second term converges to zero by Proposition 2 in
\cite{IksMoe} on stable limit for the number of renewals. In view
of Lemma \ref{lemma_distances} to prove  convergence to zero of
the first term it is sufficient to check that
$$
\lin d_{\QUpsilon}\Big(\frac{\log^2 n}{n} X_n-\log n-\log\log n,
\frac{\log^2 n}{n} N_n-\log n-\log\log n \Big)= 0,$$ for some
$\QUpsilon\in (0,1]$, which in view of the properties (Hom) and
(Aff) in Proposition \ref{was_dis_prop} amounts to the estimate
\begin{equation}\label{o_est_rv1}
 d_{\QUpsilon}(X_n,N_n)=o(n^{\QUpsilon}\log^{-2\QUpsilon}n),\;\;n\to\infty.
\end{equation}
In the like way, proving Theorem \ref{main_thm} in the case $a\in
(0,1)$ reduces to showing that
\begin{equation}\label{o_est_rva}
d_{\QUpsilon}(X_n,N_n)=o(n^{\QUpsilon/(2-a)}), \ \ n\to\infty,
\end{equation}
for some $\QUpsilon\in (0,1]$.

Using recurrences \eqref{xn_rec} for $X_n$ and \eqref{renewal_rec}
for $N_n$ we obtain
\begin{eqnarray*}
t_n&:=&d_{\QUpsilon}(X_n,N_n)=d_{\QUpsilon}(X'_{n-I_n},N'_{n-(\xi\wedge n)})\leq d_{\QUpsilon}
(N'_{n-I_n},N'_{n-(\xi\wedge n)})
+ d_{\QUpsilon}(X'_{n-I_n},N'_{n-I_n})\\
    &\leq& d_{\QUpsilon}(N'_{n-I_n},N'_{n-(\xi\wedge n)}) +
    \me|\widehat{X}_{n-I_n}-\widehat{N}_{n-I_n}|^{\QUpsilon}=:
c_n+\sum_{k=1}^{n-1}\mmp\{I_n=n-k\} \me|\widehat{X}_{k}-\widehat{N}_{k}|^{\QUpsilon},
\end{eqnarray*}
for arbitrary pairs
$\big((\widehat{X}_k,\widehat{N}_k)\big)_{1\leq k\leq n-1}$
independent of $I_n$ such that $\widehat{X}_k\od X_k$,
$\widehat{N}_k\od N_k$. Passing to the infimum over all such pairs
leads to
\begin{equation}\label{recursion_rv1}
t_n\leq c_n+\sum_{k=1}^{n-1}\mmp\{I_n=n-k\}t_k.
\end{equation}
We shall use \eqref{recursion_rv1} to estimate $t_n$.

First we find an appropriate bound for $c_n$. Let
$(\hat{I}_n,\hat{\xi})$ be a coupling of $I_n$ and $\xi$ such that
(recall (Inf) in Proposition \ref{was_dis_prop})
$d_{\QUpsilon}(I_n,\xi\wedge n)=\me|\hat{I}_n-\hat{\xi}\wedge
n|^{\QUpsilon}$. Let $\big(\hat{N}_k\big)_{k\in\mn}$ be a copy of
$\big(N_k\big)_{k\in\mn}$ independent of $(\hat{I}_n,\hat{\xi})$.
Since $\big(\hat{I}_n,\hat{\xi},\big(\hat{N}_k\big)\big)$ is a
particular coupling we have
\begin{eqnarray*}
c_n=d_{\QUpsilon}(N'_{n-I_n},N'_{n-(\xi\wedge n)})\leq\me|\hat{N}_{n-\hat{I}_n}-\hat{N}_{n-(\hat{\xi}
\wedge n)}|^{\QUpsilon}.
\end{eqnarray*}
Exploiting the stochastic inequality
$$
N_{x+y}-N_x\overset{d}{\leq} N_y,\;\;~~x,y\in{\mathbb N}
$$
yields
\begin{eqnarray*}
\me|\hat{N}_{n-\hat{I}_n}-\hat{N}_{n-\hat{\xi}\wedge n}|^{\QUpsilon}\leq\me\hat{N}^{\QUpsilon}_{|\hat{I}_n-\hat{\xi}
\wedge n|}.
\end{eqnarray*}
Furthermore, we obviously have
$N_n\leq n$, hence
$$
c_n\leq\me|\hat{I}_n-\hat{\xi}\wedge n|^{\QUpsilon}=d_{\QUpsilon}(I_n,\xi\wedge n).
$$
Now we invoke the Kantorovich-Rubinstein representation ((Rep) in Proposition \ref{was_dis_prop})
for $d_{\QUpsilon}$.
Set $\mathcal{F}_{\QUpsilon,0}:=\mathcal{F}_{\QUpsilon}\cap \{f:f(0)=0\}$ and note that $f\in\mathcal{F}_{\QUpsilon,0}$ implies $|f(x)|\leq |x|^{\QUpsilon},\;\;x\in\mr$.
We have
\begin{eqnarray*}
 c_n&\leq&d_{\QUpsilon}(I_n,\xi\wedge n)=\sup_{f\in\mathcal{F}_{\QUpsilon}}\Big|\me f(I_n) -\me f(\xi\wedge n)\Big|=\sup_{f\in\mathcal{F}_{\QUpsilon,0}}\Big|\me f(I_n) -\me f(\xi\wedge n)\Big|\\
    &=&\sup_{f\in\mathcal{F}_{\QUpsilon,0}}\Big|\sum_{k=1}^{n-1}\mmp\{I_n=k\}f(k) -\sum_{k=1}^{n-1}
    \mmp\{\xi=k\}f(k)-f(n)\sum_{k\geq n}\mmp\{\xi=k\}\Big|\\
    &\leq&\sum_{k=1}^{n-1}\Big|\mmp\{I_n=k\}-\mmp\{\xi=k\}\Big|k^{\QUpsilon} +n^{\QUpsilon}\mmp\{\xi\geq n\}.
\end{eqnarray*}

For appropriate $\QUpsilon\in(0,1]$ (to be specified below) such
that $a+\QUpsilon>1$ use Lemma \ref{pseudo_moments} in the
Appendix along with the relation $\mmp\{\xi\geq n\}=O(n^{a-2})$ to
obtain the
estimate $c_n=O(n^{\QUpsilon+a-2})$. With this bound for $c_n$ a
$O$-estimate for $t_n$ follows using Lemma \ref{boundedness}.

If $a\in (0,1)$ one can take $q=1$. Then the cited lemma applies
with $\psi_n=n$ and $r_n=Mn^{a-1}$ ($M$ large enough) and gives
estimate
$$
d_{\QUpsilon}(X_n,N_n)=O(n^a),
$$
which implies \eqref{o_est_rva}.

For the case $a=1$ application of the same lemma
with $\psi_n=n/(\log (n+1))$ and $r_n=Mn^{\QUpsilon-1}$ ($M$ large
enough) leads to $t_n\leq Mn^{\QUpsilon}(\log n)^{-1}$. Thus
\eqref{o_est_rv1} holds for $\QUpsilon\in(0,1/2)$. The proof is
complete.

\subsection{Proof of Corollaries \ref{tbl_cor} and  \ref{seg_cites_cor}.}
We follow closely the proofs of Theorem 5.2 and Corollary 6.2 in \cite{DIMR1}.
In view of
$$
\frac{b\log^2 n}{n} L_n-\log n-\log\log n =\frac{\log^2 n}{n}
X_n-\log n-\log\log n +\frac{\log^2 n}{n}\Big(bL_n-X_n\Big),
$$
it is enough to show that $\frac{\log^2 n}{n}\Big(bL_n-X_n\Big)\to
0$ in $L_2$.  

Let $T_j$'s be independent exponential variables with rates
$\lambda_j, j\geq 2$.
Assuming the $T_j$'s
independent of the sequence of
states visited by $\Pi_n$ we may identify $T_j$
with  the time $\Pi_n$  spends in the state $j$ provided this state is visited.
Given the sequence of visited states is
$n =
i_0
> i_1
> \cdots
> i_{k-1}
> i_k = 1$ the total branch length $L_n$ is distributed like
$\sum_{r=0}^{k-1} i_r T_{i_r}$ for $n\in\mn\backslash\{1\}$.

For $k \in \{1,\ldots,n\}$ and ${\rm i} = (i_0,\ldots,i_k)$ with
$n = i_0 > i_1 > \cdots > i_{k-1} > i_k = 1$ define the events
$A_{k,\,{\rm i}}:=\{X_n = k, (\Pi_n(t_0),\ldots,\Pi_n(t_k) ) =
{\rm i}\}$, where $t_0=0$ and $t_1<t_2<\ldots$ are the collision
epochs. We have
\begin{eqnarray*}
\me (bL_n - X_n)^2 &=&
\sum_{k,\,{\rm i}}\mmp\{A_{k,\,{\rm i}}\}\me
\Big(\sum_{r=0}^{k-1}(b i_rT_{i_r}-1)\Big)^2\\
&=&\sum_{k,\,{\rm i}}\mmp\{A_{k,\,{\rm i}}\}\Big(
\sum_{r=0}^{k-1}\me(bi_rT_{i_r}-1)^2 + \sum_{r,s=0,r\neq
s}^{k-1}\me(bi_rT_{i_r}-1)(bi_sT_{i_s}-1)\Big)
\end{eqnarray*}
Furthermore,  $\lambda_n=bn+O(\log n)$ as $n\to\infty$ for $a=1$
and $b>0$ (see \eqref{total_rates_1}) which implies
$$
|\me (bkT_k-1)| = O(k^{-1}\log k)\text{  and  }\me (bkT_k-1)^2 = 1+O(k^{-1}\log k).
$$
Therefore,
\begin{eqnarray*}
\me (bL_n - X_n)^2 &\leq& \sum_{k,\,{\rm i}}\mmp\{A_{k,\,{\rm
i}}\}\Big( \sum_{r=2}^n \me(brT_{r}-1)^2 +
\Big(\sum_{r=2}^n |\me(brT_{r}-1)|\Big)^2\Big)\\
&=&\sum_{k,\,{\rm i}}\mmp\{A_{k,\,{\rm i}}\}\Big(n + O(\log^4
n)\Big)=n+O(\log^4 n),
\end{eqnarray*}
and the convergence in $L_2$ follows.

Corollary \ref{seg_cites_cor} follows from the fact that given
$L_n$ the distribution of $M_n$ is Poisson with mean $r L_n$.
See Corollary 6.2 in \cite{DIMR1} for details.
\subsection{Proofs of Theorem \ref{main1} and Corollary \ref{centralmoments}}
  Let us verify (\ref{xnmomentasy}) by induction on $j\in\mn$. From
   (\ref{xn_rec}) it follows that $a_1:=\me X_1=0$ and
   $a_n:=\me X_n=1 + \sum_{m=2}^{n-1}p^{(1)}_{n,m}a_m$, $n\in\mn\setminus\{1\}$.
   In the following we apply the method of sequential approximations to the sequence
   $(a_n)_{n\in\mn}$. The sequence $(b_n)_{n\in\mn}$, defined
   via $b_1:=0$ and $b_n:=a_n-n/\log n$ for $n\in\mn\setminus\{1\}$,
   satisfies the recursion
   \begin{eqnarray*}
      b_n
      & = & a_n - \frac{n}{\log n}
      \ = \ 1 + \sum_{m=2}^{n-1}p^{(1)}_{n,m}\bigg(\frac{m}{\log m}+b_m\bigg) - \frac{n}{\log n}
      \ = \ q_n + \sum_{m=2}^{n-1}p^{(1)}_{n,m}b_m,
   \end{eqnarray*}
   $n\in\mn\setminus\{1\}$, where
   $q_n:=1-n/\log n+\sum_{m=2}^{n-1}p^{(1)}_{n,m}m/\log m$,
   $n\in\mn\setminus\{1\}$.
   By Corollary \ref{appmain} (applied with $\alpha:=1$ and $p:=1$),
   $$
   q_n\ =\ 1 - \frac{n}{\log n} + \bigg(\frac{n}{\log n}-1 + \frac{m_1}{\log n}+O\bigg(\frac{1}{\log^2n}\bigg)\bigg)
   \ =\ \frac{m_1}{\log n} + O\bigg(\frac{1}{\log^2n}\bigg),
   $$
   where $m_1:=c_{b,1,1}=2+\Psi(b)$. The sequence
   $(c_n)_{n\in\mn}$, defined via $c_1:=0$ and $c_n:=b_n-m_1n/\log^2n$
   for $n\in\mn\setminus\{1\}$, therefore satisfies the recursion
   $$
   c_n
   \ = \ b_n - m_1\frac{n}{\log^2n}
   \ = \ q_n + \sum_{m=2}^{n-1}p^{(1)}_{n,m}\bigg(m_1\frac{m}{\log^2m}+c_m\bigg) - m_1\frac{n}{\log^2n}
   \ = \ q_n' + \sum_{m=2}^{n-1}p^{(1)}_{n,m}c_m,
   $$
   $n\in\mn\setminus\{1\}$,
   where $q_n':=q_n - m_1n/\log^2n + m_1\sum_{m=2}^{n-1}p^{(1)}_{n,m}m/\log^2m$,
   $n\in\mn\setminus\{1\}$.
   By Corollary \ref{appmain} (applied with $\alpha:=1$ and $p:=2$),
   \begin{eqnarray*}
      q_n'
      & = & q_n - m_1\frac{n}{\log^2n} + m_1\bigg(\frac{n}{\log^2n}-\frac{1}{\log n}+O\bigg(\frac{1}{\log^2n}\bigg)\bigg)
      \ = \ O\bigg(\frac{1}{\log^2n}\bigg),
   \end{eqnarray*}
   since $q_n=m_1/\log n+O(1/\log^2n)$. By Lemma \ref{appstop}
   (applied with $\alpha:=1$ and $p:=3$), it follows that
   $c_n=O(n/\log^3n)$. Thus, (\ref{xnmomentasy}) holds for $j=1$.
   Assume now that $j\ge 2$. From
   $\me X_{I_n}^j=\me(X_n-1)^j
   =\sum_{i=0}^{j-1}\binom{j}{i}
   (-1)^{j-i}\me X_n^i + \me X_n^j
   $
   it follows that
   \begin{eqnarray*}
      a_{n,j}\ :=\ \me X_n^j
      & = & \sum_{i=0}^{j-1}
            \binom{j}{i}
            (-1)^{j-1-i}\me X_n^i + \me X_{I_n}^j
      \ = \ q_{n,j} + \sum_{m=2}^{n-1}p^{(1)}_{n,m}a_{m,j},
   \end{eqnarray*}
   $n\in\mn\setminus\{1\}$,
   where
   $q_{n,j}:=\sum_{i=0}^{j-1}
   \binom{j}{i}
   (-1)^{j-1-i}\me X_n^i$,
   $n\in\mn\setminus\{1\}$. Since, by induction, for all $i<j$,
   $$
   \me X_n^i\ =\ \frac{n^i}{\log^in}
   \bigg(1+\frac{m_i}{\log n}+O\bigg(\frac{1}{\log^2n}\bigg)\bigg),
   $$
   it follows that 
   (the summand for $i=j-1$ asymptotically dominates the others)
   $$
   q_{n,j}\ =\ \frac{jn^{j-1}}{\log^{j-1}n}\bigg(1 + \frac{m_{j-1}}{\log n}
   + O\bigg(\frac{1}{\log^2n}\bigg)\bigg).
   $$
   Now apply the method of sequential approximations to the sequence $(a_{n,j})_{n\in\mn}$.
   The sequence $(b_{n,j})_{n\in\mn}$, defined via $b_{1,j}:=0$ and
   $b_{n,j}:=a_{n,j}-n^j/\log^jn$ for $n\in\{2,3,\ldots\}$, satisfies
   the recursion
   $$
   b_{n,j}\ =\ q_{n,j}' + \sum_{m=2}^{n-1}p^{(1)}_{n,m} b_{m,j},\quad n\in\{2,3,\ldots\},
   $$
   where $q_{n,j}':=q_{n,j} - n^j/\log^jn + \sum_{m=2}^{n-1}p^{(1)}_{n,m}m^j/\log^jm$,
   $n\in\{2,3,\ldots\}$. By Corollary \ref{appmain} (applied with $\alpha:=j$
   and $p:=j$),
   \begin{eqnarray*}
      q_{n,j}'
      & = & j\frac{n^{j-1}}{\log^{j-1}n} + jm_{j-1}\frac{n^{j-1}}{\log^jn}
            + O\bigg(\frac{n^{j-1}}{\log^{j+1}n}\bigg) - \frac{n^j}{\log^jn}\\
      &   & \hspace{1cm}
            + \frac{n^j}{\log^jn} - j\frac{n^{j-1}}{\log^{j-1}n}
            + \kappa_j\frac{n^{j-1}}{\log^jn} + O\bigg(\frac{n^{j-1}}{\log^{j+1}n}\bigg)\\
      & = & jm_j\frac{n^{j-1}}{\log^jn}
            + O\bigg(\frac{n^{j-1}}{\log^{j+1}n}\bigg),
   \end{eqnarray*}
   where $\kappa_j:=c_{b,j,j}$ and $m_j:=m_{j-1}+\kappa_j/j$.
   The sequence $(c_{n,j})_{n\in\mn}$, defined via $c_{1,j}:=0$ and
   $c_{n,j}:=b_{n,j}-m_jn^j/\log^{j+1}n$ for $n\in\{2,3,\ldots\}$,
   therefore satisfies the recursion
   $$
   c_{n,j}\ =\ q_{n,j}'' + \sum_{m=2}^{n-1}p^{(1)}_{n,m}c_{m,j},
   \quad n\in\{2,3,\ldots\},
   $$
   where $q_{n,j}'':=q_{n,j}' - m_jn^j/\log^{j+1}n
   + m_j\sum_{m=2}^{n-1}p^{(1)}_{n,m}m^j/\log^{j+1}m$,
   $n\in\{2,3,\ldots\}$.
   By Corollary \ref{appmain} (applied with $\alpha:=j$ and $p:=j+1$),
   \begin{eqnarray*}
      q_{n,j}''
      & = & jm_j\frac{n^{j-1}}{\log^jn} + O\bigg(\frac{n^{j-1}}{\log^{j+1}n}\bigg)
            -m_j\frac{n^j}{\log^{j+1}n}\\
      &   & \hspace{1cm}
            + m_j\bigg(\frac{n^j}{\log^{j+1}n}-j\frac{n^{j-1}}{\log^jn}+O\bigg(\frac{n^{j-1}}{\log^{j+1}n}\bigg)\bigg)
      \ = \ O\bigg(\frac{n^{j-1}}{\log^{j+1}n}\bigg).
   \end{eqnarray*}
   By Lemma \ref{appstop} (applied with $\alpha:=j$ and $p:=j+2$),
   it follows that $c_{n,j}=O(n^j/\log^{j+2}n)$, which shows that
   (\ref{xnmomentasy}) holds for $j$. The induction is
   complete which finishes the proof of Theorem \ref{main1}.

   We now turn to the proof of Corollary \ref{centralmoments}. Let us
   first verify that the sequence $(m_j)_{j\in\mn_0}$, recursively defined
   in Theorem \ref{main1}, satisfies the inversion formula
   \begin{equation} \label{inversion}
      \sum_{i=0}^j \binom{j}{i}(-1)^{j-i}m_i\ =\ \frac{(-1)^j}{j}\Be(b,j-1),
      \quad j\in\mn\setminus\{1\}.
   \end{equation}
   Using the formula
   $\Psi(x+1)=\Psi(x)+1/x$, $x\in (0,\infty)$, it is readily checked
   that $\kappa_{j+1}-\kappa_j=2+\Psi(b+j)$, $j\in\mn_0$. For all $j\in\mn_0$
   it follows that $\kappa_j=\sum_{i=0}^{j-1}(\kappa_{i+1}-\kappa_i)
   =\sum_{i=0}^{j-1}(2+\Psi(b+i))=2j + \sum_{i=0}^{j-1}\Psi(b+i)$ and
   \begin{equation} \label{mj}
      m_j
       = \sum_{l=1}^j (m_l-m_{l-1})
       = \sum_{l=1}^j \frac{\kappa_l}{l}
       = \sum_{l=1}^j \bigg(2+\frac{1}{l}\sum_{i=0}^{l-1}\Psi(b+i)\bigg)\\
       = 2j + \sum_{i=0}^{j-1}\Psi(b+i)\sum_{l=i+1}^j\frac{1}{l}.
   \end{equation}
   By (\ref{mj}), for $j\in\{2,3,\ldots\}$,
   \begin{eqnarray*}
      &   & \hspace{-10mm}\sum_{i=0}^j \binom{j}{i}(-1)^{j-i}m_i
      \ = \ \sum_{i=1}^j \binom{j}{i}(-1)^{j-i}
            \bigg(2i+\sum_{k=0}^{i-1}\Psi(b+k)\sum_{l=k+1}^i\frac{1}{l}\bigg)\\
      & = & \sum_{i=1}^j\binom{j}{i}(-1)^{j-i}\sum_{k=0}^{i-1}\Psi(b+k)\sum_{l=k+1}^i\frac{1}{l}\\
      & = & \sum_{k=0}^{j-1}\Psi(b+k)\sum_{l=k+1}^j\frac{1}{l}\sum_{i=l}^j \binom{j}{i}(-1)^{j-i}
      \ = \ \sum_{k=0}^{j-1}\Psi(b+k)\sum_{l=k+1}^j\frac{1}{l}\binom{j-1}{l-1}(-1)^{j-l}\\
      & = & \frac{1}{j}\sum_{k=0}^{j-1}\Psi(b+k)\sum_{l=k+1}^j \binom{j}{l}(-1)^{j-l}
      \ = \ \frac{1}{j}\sum_{k=0}^{j-1}\Psi(b+k)\binom{j-1}{k}(-1)^{j-1-k}.
   \end{eqnarray*}
   Plugging in $\binom{j-1}{k}=\binom{j-2}{k-1}+\binom{j-2}{k}$ and
   reordering with respect to $\binom{j-2}{k}$ leads to
   \begin{eqnarray*}
      \sum_{i=0}^j\binom{j}{i}(-1)^{j-i}m_i
      & = & \frac{1}{j}\sum_{k=0}^{j-2} (-1)^{j-2-k}
            \binom{j-2}{k}(\Psi(b+k+1)-\Psi(b+k))\\
      & = & \frac{(-1)^j}{j}\sum_{k=0}^{j-2} (-1)^k\binom{j-2}{k}\frac{1}{b+k}
      \ = \ \frac{(-1)^j}{j}\Be(b,j-1),
   \end{eqnarray*}
   where the last equality holds, since
   $\sum_{k=0}^n(-1)^k\binom{n}{k}/(b+k)=\Be(b,n+1)$
   for all $n\in\mn_0$, which is for example readily verified by induction
   on $n\in\mn_0$. Thus, (\ref{inversion}) is established.

   Thanks to Theorem \ref{main1} and the inversion formula (\ref{inversion})
   the proof of Corollary \ref{centralmoments} is now straightforward.
   Basically the same argument has for example been used by Panholzer
   \cite[p.~277]{Pan2}. Plugging in the expansion (\ref{xnmomentasy}) for
   the ordinary moments shows that
   \begin{eqnarray*}
      &   & \hspace{-10mm}
            \me(X_n-\me X_n)^j
      \ = \ \sum_{i=0}^j \binom{j}{i}(-1)^{j-i}\me X_n^i (\me X_n)^{j-i}\\
      & = & \sum_{i=0}^j \binom{j}{i}(-1)^{j-i}
            \frac{n^i}{\log^in}
            \bigg(1+\frac{m_i}{\log n}+O\bigg(\frac{1}{\log^2n}\bigg)\bigg)
            \bigg(\frac{n}{\log n}
            \bigg(1+\frac{m_1}{\log n}+O\bigg(\frac{1}{\log^2n}\bigg)\bigg)\bigg)^{j-i}\\
      & = & \frac{n^j}{\log^jn}\sum_{i=0}^j \binom{j}{i}(-1)^{j-i}
            \bigg(1 + \frac{m_i}{\log n}+O\bigg(\frac{1}{\log^2n}\bigg)\bigg)
            \bigg(1+\frac{(j-i)m_1}{\log n}+O\bigg(\frac{1}{\log^2n}\bigg)\bigg)\\
      & = & \frac{n^j}{\log^jn}
            \sum_{i=0}^j \binom{j}{i}(-1)^{j-i}
            \bigg(1+\frac{(j-i)m_1+m_i}{\log n}+O\bigg(\frac{1}{\log^2n}\bigg)\bigg)\\
      & = & \frac{n^j}{\log^jn}\sum_{i=0}^j \binom{j}{i}(-1)^{j-i}
            + \frac{n^j}{\log^{j+1}n}\sum_{i=0}^j\binom{j}{i}
            (-1)^{j-i}((j-i)m_1+m_i) + O\bigg(\frac{n^j}{\log^{j+2}n}\bigg)\\
      & = & \frac{n^j}{\log^{j+1}n}\frac{(-1)^j}{j}\Be(b,j-1)
            + O\bigg(\frac{n^j}{\log^{j+2}n}\bigg),
   \end{eqnarray*}
   since, for $j\ge 2$,
   $\sum_{i=0}^j\binom{j}{i}(-1)^{j-i}=0$,
   $\sum_{i=0}^j\binom{j}{i}(-1)^{j-i}(j-i)=0$, and
   $\sum_{i=0}^j\binom{j}{i}(-1)^{j-i}m_i=(-1)^j/j\Be(b,j-1)$
   by (\ref{inversion}). The proof of Corollary \ref{centralmoments}
   is complete.

\section{Appendix}
For each $n\in\mn$ let $(p_{n,\,k})_{0\le k\le n}$ be
an arbitrary probability distribution with $p_{n,\,n}<1$. Define a sequence
$(a_n)_{n\in\mn}$ as a (unique) solution to the recursion
\begin{equation}\label{recursion}
a_n=r_n+\sum_{k=0}^n p_{n,k}a_k,\quad n\in\mn,
\end{equation}
with given $r_n\ge 0$ and given initial value $a_0=a\ge 0$. The
following result is Lemma 6.1 from \cite{GIM_coal}.
\begin{lemma}\label{L61}
\label{boundedness} Suppose there exists a sequence
$(\psi_n)_{n\in\mn}$ such that
\begin{itemize}
\item[\rm(C1)]
$\liminf_{n\to\infty}\psi_n\sum_{k=0}^{n}(1-k/n)p_{n,k}>0$,
\item[\rm(C2)] the sequence $(r_k\psi_k/k)_{k\in\mn}$ is non-increasing.
\end{itemize}
Then ${a_n}$, defined by \eqref{recursion}, satisfies
\begin{equation}\label{bounded}
a_n=O\Big(\sum_{k=1}^{n}\frac{r_k\psi_k}{k}\Big),\quad n\to\infty.
\end{equation}
\end{lemma}

\begin{lemma} \label{appstop}
   Let $(a_n)_{n\in\mn}$ be a sequence of real numbers satisfying
   the recursion $a_1=0$ and $a_n=q_n+\sum_{m=2}^{n-1} p^{(1)}_{n,m} a_m$,
   $n\in\mn\setminus\{1\}$, for some given sequence
   $(q_n)_{n\in\mn\setminus\{1\}}$. If $q_n=O(n^{\alpha-1}/\log^{p-1}n)$
   for some given constants $\alpha\in (0,\infty)$ and $p\in [0,\infty)$,
   then $a_n=O(n^\alpha/\log^pn)$.
\end{lemma}
\begin{proof}
   Fix some $\delta$ such that $0<\delta<\alpha$.
   Set $a'_n:=|a_n|/n^\delta$ and $q'_n:=|q_n|/n^\delta$.
   Then $q'_n\leq M n^{\alpha-1-\delta}/\log^{p-1} n=:r_n$
   for some $M>0$ and all $n\geq 2$. Further,
   \begin{eqnarray*}
      a'_n
      & \le & q'_n + \sum_{m=2}^{n-1} p^{(1)}_{n,m}\frac{|a_m|}{n^\delta}
      \ \le\ q'_n + \sum_{m=2}^{n-1} p^{(1)}_{n,m}\frac{|a_m|}{m^\delta}\\
      & = & q'_n + \sum_{m=2}^{n-1} p^{(1)}_{n,m} a'_m
      \ \le \ r_n + \sum_{m=2}^{n-1} p^{(1)}_{n,m} a'_m.
   \end{eqnarray*}
   Set $\psi_n:=n/\log n$, then both conditions (C1) and (C2) of Lemma
   \ref{L61} are fulfilled. Hence $a'_n=O(\sum_{k=2}^n
   k^{\alpha-1-\delta}/\log^p k)=O(n^{\alpha-\delta}/\log^p n)$ and
   $|a_n|=n^\delta a'_n=O(n^{\alpha}/\log^p n)$.
\end{proof}
\begin{lemma}\label{pseudo_moments}
For the first decrement $I_n$ of the Markov chain $\big(\Pi_n\big)$
associated with the $\be\,(a,b)$-coalescent ($a\in (0,1]$ and
$b>0$) and a random variable $\xi$ with distribution $(p^{(a)}_k)_{k\in\mn}$
\begin{equation}\label{main_est_claim}
\sum_{k=1}^{n-1}k^{\QUpsilon}|\mmp\{I_n=k\}-\mmp\{\xi=k\}|=O(n^{a+\QUpsilon-2}),\;
\end{equation}
whenever $0<\QUpsilon\leq 1$ and $\QUpsilon+a>1$.
\end{lemma}
\begin{proof}
For the $\be\,(a,b)$-coalescents formula \eqref{lrates} reads
\begin{equation*}\label{rates_beta}
\lambda_{n,\,k+1}=\int_{0}^1 x^{k-1}(1-x)^{n-k-1}\Lambda({\rm
d}x)=\frac{\Be(a+k-1,n-k+b-1)}{\Be(a,b)}.
\end{equation*}
Using the known estimate for the gamma function (see formula
(6.1.47) in \cite{Abramovitz_Stegun})
$$
\Big|\frac{\Gamma(n+c)}{\Gamma(n+d)}-n^{c-d}\Big|\leq
M_{c,\,d}n^{c-d-1},\;\;n\geq 2,\;\;c,d>-2,
$$
we obtain
\begin{eqnarray}
\binom{n}{k+1}
\lambda_{n,\,k+1}&=&\binom{n}{k+1}\frac{\Be(a+k-1,n-k+b-1)}{\Be(a,b)}\label{total_rates_b}\\
&=&\frac{\Gamma(n+1)\Gamma(a+k-1)\Gamma(n-k+b-1)}{\Gamma(k+2)\Gamma(n-k)\Gamma(n+a+b-2)\Be(a,b)}\nonumber\\
&=&\frac{\Gamma(a+k-1)}{(k+1)!\Be(a,b)}\Big(n^{3-a-b}+O(n^{2-a-b})\Big)\Big((n-k)^{b-1}+O\Big((n-k)^{b-2}\Big)\Big)\nonumber,
\end{eqnarray}
uniformly for $1 \leq k\leq n-1$ and $n \geq 2$.


Using \eqref{total_rates} with $\Lambda$ given by \eqref{beta} we
infer (see also Corollary 2 in \cite{GneYak})
\begin{equation}\label{total_rates_a}
\lambda_{n}=\frac{\Gamma(a)}{(2-a)\Be(a,b)}n^{2-a}+O(n^{1-a})=\frac{\Gamma(a)}{(2-a)\Be(a,b)}n^{2-a}\Big(1+O(n^{-1})\Big),
\end{equation}
when $a\in (0,1)$ and $b>0$, and
\begin{equation}\label{total_rates_1}
\lambda_{n}=bn+O(\log n),
\end{equation}
when $a=1$ and $b>0$. Hence for $0<a<1$, $b>0$, $n\geq 2$ and
$k=1,\ldots,n-1$
\begin{eqnarray*}
p^{(a)}_{n,n-k}&=&\frac{(2-a)\Gamma(a+k-1)}{\Gamma(a)(k+1)!}n^{1-b}\Big((n-k)^{b-1}+O\Big((n-k)^{b-2}\Big)\Big)\Big(1+O(n^{-1})\Big)\\
& = & p^{(a)}_k \Big((1-k/n)^{b-1}+O\Big(n^{-1}(1-k/n)^{b-2}\Big)\Big)\Big(1+O(n^{-1})\Big)\\
& = & p^{(a)}_k \Big((1-k/n)^{b-1}+O\Big(n^{-1}(1-k/n)^{b-2}\Big)\Big)\\
& = & p^{(a)}_k(1-k/n)^{b-1}+O\Big(p^{(a)}_k n^{-1}(1-k/n)^{b-2}\Big).
\end{eqnarray*}
Analogously for $a=1$
\begin{eqnarray*}
p^{(1)}_{n,n-k}
& = & p^{(1)}_k\Big((1-k/n)^{b-1}+O\Big(n^{-1}(1-k/n)^{b-2}\Big)\Big)\Big(1+O(n^{-1}\log n)\Big)\\
& = & p^{(1)}_k\Big((1-k/n)^{b-1}+O\Big(n^{-1}(1-k/n)^{b-2}\Big)+O\Big(n^{-1}\log n(1-k/n)^{b-1}\Big)\Big)\\
& = & p^{(1)}_k(1-k/n)^{b-1}+O\Big(p^{(1)}_k n^{-1}(1-k/n)^{b-2}\Big)+O\Big(p^{(1)}_kn^{-1}\log n(1-k/n)^{b-1}\Big).
\end{eqnarray*}
Substituting these expansions into the left-hand side of
\eqref{main_est_claim} gives
\begin{eqnarray*}
\sum_{k=1}^{n-1}k^{\QUpsilon}\Big|\mmp\{I_n=k\}-\mmp\{\xi=k\}\Big|&\leq &\sum_{k=1}^{n-1}p^{(a)}_k k^{\QUpsilon}\Big|\Big(1-\frac{k}{n}\Big)^{b-1}-1\Big|\\
&+&\frac{c_1}{n}\sum_{k=1}^{n-1}p^{(a)}_k k^{\QUpsilon}\Big(1-\frac{k}{n}\Big)^{b-2}\\
&=:&S_1(a,n)+S_2(a,n),
\end{eqnarray*}
for $0<a<1$, and
\begin{eqnarray*}
\sum_{k=1}^{n-1}k^{\QUpsilon}\Big|\mmp\{I_n=k\}-\mmp\{\xi=k\}\Big|&\leq& S_1(1,n)+S_2(1,n)+\frac{c_2\log n}{n}\sum_{k=1}^{n-1}p^{(1)}_k k^{\QUpsilon}\Big(1-\frac{k}{n}\Big)^{b-1}\\
&=:&S_1(1,n)+S_2(1,n)+S_3(1,n),
\end{eqnarray*}
for $a=1$. Here and hereafter $c_1,c_2,\ldots$ denote some
positive constants whose values are of no importance. Our aim is
to show that $S_i(a,n)=O(n^{\QUpsilon+a-2})$ for $i=1,2$ and
$S_3(1,n)=O(n^{\QUpsilon-1})$. By virtue of $p^{(a)}_k\leq
c_3k^{a-3}$, for all $k\in\mn$, we infer
\begin{eqnarray*}
S_1(a,n)&\leq&c_3\sum_{k=1}^{n-1}k^{a+\QUpsilon-3}\Big|\Big(1-\frac{k}{n}\Big)^{b-1}-1\Big|\\
&=&c_3\sum_{k=1}^{[n/2]}k^{a+\QUpsilon-3}\Big|\Big(1-\frac{k}{n}\Big)^{b-1}-1\Big|+c_3\sum_{k=[n/2]+1}^{n-1}k^{a+\QUpsilon-3}\Big|\Big(1-\frac{k}{n}\Big)^{b-1}-1\Big|\\
&\leq&\frac{c_4}{n}\sum_{k=1}^{[n/2]}k^{a+\QUpsilon-2}+c_3n^{a+\QUpsilon-2}\Big(\frac{1}{n}\sum_{k=[n/2]+1}^{n-1}\Big(\frac{k}{n}\Big)^{a+\QUpsilon-3}\Big|\Big(1-\frac{k}{n}\Big)^{b-1}-1\Big| \Big),
\end{eqnarray*}
where the inequality $|(1-x)^{\QUpsilon}-1|\leq c_5x,\;\;
x\in[0,1/2]$ has been utilized. The expression in the parentheses
converges to $\int_{1/2}^1 x^{a+\QUpsilon-3}|(1-x)^{b-1}-1|{\rm
d}x<\infty$. Hence $S_1(a,n)=O(n^{\QUpsilon+a-2})$.

Similarly
\begin{eqnarray*}
S_2(a,n)\hspace{-1cm}&&\\
&&\leq\frac{c_6}{n}\sum_{k=1}^{n-1}k^{a+\QUpsilon-3}\Big(1-\frac{k}{n}\Big)^{b-2}\\
&&=\frac{c_6}{n}\sum_{k=1}^{[n/2]}k^{a+\QUpsilon-3}\Big(1-\frac{k}{n}\Big)^{b-2}+\frac{c_6}{n}\sum_{k=[n/2]+1}^{n-1}k^{a+\QUpsilon-3}\Big(1-\frac{k}{n}\Big)^{b-2}\\
&&\leq\frac{c_6}{n}\sum_{k=1}^{[n/2]}k^{a+\QUpsilon-3}\Big(1-\frac{k}{n}\Big)^{b-2}+c_6\sum_{k=[n/2]+1}^{n-1}k^{a+\QUpsilon-3}\Big(1-\frac{k}{n}\Big)^{b-1}\\
&&=\frac{c_6}{n}\sum_{k=1}^{[n/2]}k^{a+\QUpsilon-3}\Big(1-\frac{k}{n}\Big)^{b-2}+c_6n^{a+\QUpsilon-2}\Big(\frac{1}{n}\sum_{k=[n/2]+1}^{n-1}\Big(\frac{k}{n}\Big)^{a+\QUpsilon-3}\Big(1-\frac{k}{n}\Big)^{b-1}\Big)\\
&&=O(n^{a+\QUpsilon-2})
\end{eqnarray*}
since the first term is $O(n^{-1})$ and the second term is
$O(n^{a+\QUpsilon-2})$ by the same reasoning as for $S_1(a,n)$.

Finally
\begin{eqnarray*}
S_3(1,n)&\leq &\frac{c_7\log n}{n}\sum_{k=1}^{n-1}k^{\QUpsilon-2}\Big(1-\frac{k}{n}\Big)^{b-1}\\
&\leq&\frac{c_7\log n}{n}\sum_{k=1}^{n-1}k^{\QUpsilon-2}\Big|\Big(1-\frac{k}{n}\Big)^{b-1}-1\Big|+\frac{c_7\log n}{n}\sum_{k=1}^{n-1}k^{\QUpsilon-2}\\
&=&O(n^{\QUpsilon-2}\log n) + O(n^{-1}\log n),
\end{eqnarray*}
in view of the estimate for $S_1(a,n)$. Therefore $S_3(1,n)=O(n^{\QUpsilon-1})$ and the proof is complete.
\end{proof}
We provide a basic lemma concerning the total rates
of the $\be\,(1,b)$-coalescent.
\begin{lemma} \label{totalrates}
   The total rates $\lambda_n$, $n\in\mn$, of the $\be\,(1,b)$-coalescent
   are explicitly given by
   \begin{equation} \label{gnexplicit}
      \lambda_n
      \ =\ b\sum_{k=1}^{n-1}\frac{k}{b+k-1}
      \ =\ b(n-1)-b(b-1)(\Psi(n+b-1)-\Psi(b)),\quad n\in\mn.
   \end{equation}
   Moreover, the total rates have the
   asymptotic expansion
   \begin{equation} \label{gnasy}
      \lambda_n\ =\ bn -b(b-1)\log n -b+b(b-1)\Psi(b) + O(n^{-1}),
      \quad n\to\infty,
   \end{equation}
   and the inverse of the total rate $\lambda_n$ has
   the asymptotic expansion
   \begin{equation} \label{gninverseasy}
      \frac{1}{\lambda_n}\ =\ \frac{1}{bn}\bigg(1 + (b-1)\frac{\log n}{n}
      +\frac{1-(b-1)\Psi(b)}{n} + O\bigg(\frac{\log^2n}{n^2}\bigg)\bigg),
      \quad n\to\infty.
   \end{equation}
\end{lemma}
\begin{proof}
   Eq.~(\ref{gnexplicit}) is well known (see, for example,
   \cite[Appendix, Eq.~(19)]{HuiMoe}. The expansion (\ref{gnasy})
   follows directly from (\ref{gnexplicit}), since $\Psi(n+b-1)
   =\log n+O(n^{-1})$ as $n\to\infty$. The last assertion (\ref{gninverseasy}) follows from
   \begin{eqnarray}
      &   & \hspace{-2cm}
            \frac{bn}{\lambda_n} - 1 - (b-1)\frac{\log n}{n} - \frac{1-(b-1)\Psi(b)}{n}\nonumber\\
      & = & \frac{bn^2 - \lambda_n\big(n+(b-1)\log n + 1 - (b-1)\Psi(b)\big)}{n\lambda_n} \label{local}\\
      & = & \frac{O(\log^2n)}{n\lambda_n}
      \ = \ O\bigg(\frac{\log^2n}{n^2}\bigg),\nonumber
   \end{eqnarray}
   where the very last equality holds, since $\lambda_n\sim bn$, and the
   equality before follows by plugging in (\ref{gnasy}) for the term
   $\lambda_n$ occurring in the numerator of the fraction in
   (\ref{local}) and multiplying everything out.
\end{proof}
The next Lemma \ref{applemmasimple} provides an asymptotic expansion as
$n\to\infty$ for sums of the form
$$
\sum_{m=2}^{n-1}\frac{m^\alpha}{(n-m)(n-m+1)\log^pm},
\qquad\alpha\in\mr, p\in [0,\infty).
$$
For parameters $\alpha>0$ we will need an even sharper version (see
Lemma \ref{applemma} below), but we start with this simpler version,
which holds for arbitrary $\alpha\in\mr$. Given the overlap with the
proof of the following Lemma \ref{applemma} and given the fact that
the proof is considerably simpler than that of Lemma \ref{applemma},
the proof of Lemma \ref{applemmasimple} is omitted.
\begin{lemma} \label{applemmasimple}
   For $\alpha\in\mr$ and $p\in [0,\infty)$, as $n\to\infty$,
   $$
   \sum_{m=2}^{n-1}\frac{m^\alpha}{(n-m)(n-m+1)\log^pm}
   \ =\ \frac{n^\alpha}{\log^pn}\bigg(1-\alpha\frac{\log n}{n} + O\bigg(\frac{1}{n}\bigg)\bigg).
   $$
\end{lemma}
The following Lemma \ref{applemma} is a sharper version of
Lemma \ref{applemmasimple} with the cost that it holds only for
$\alpha>0$. It will turn out (see the following
Corollary \ref{appmain} and the proof of Theorem \ref{main1})
that the expansion in Lemma \ref{applemma} is fundamental for the analysis
of the moments of the number of collisions of the
$\be\,(1,b)$-coalescent.
\begin{lemma} \label{applemma}
   For $\alpha\in (0,\infty)$ and $p\in [0,\infty)$,
   $$
   \sum_{m=2}^{n-1}\frac{m^\alpha}{(n-m)(n-m+1)\log^pm}
   \ =\ \frac{n^\alpha}{\log^pn}
   \bigg(
      1 - \alpha\frac{\log n}{n}
      + \frac{\alpha\Psi(\alpha)+p}{n}
      + O\bigg(\frac{1}{n\log n}\bigg)
   \bigg)
   $$
   as $n\to\infty$.
\end{lemma}
\begin{proof}
   Note first that
   \begin{eqnarray*}
      &   & \hspace{-8mm}\sum_{m=2}^{n-1}
            \frac{m^\alpha}{(n-m)(n-m+1)}\bigg(\frac{1}{\log^pn}+p\frac{-\log(m/n)}{\log^{p+1}n}\bigg)\\
      & = & \frac{1}{\log^pn}\sum_{m=2}^{n-1}\frac{m^\alpha}{(n-m)(n-m+1)}
            + \frac{p}{\log^{p+1}n}\sum_{m=2}^{n-1}
            \frac{m^\alpha(-\log(m/n))}{(n-m)(n-m+1)}\\
      & = & \frac{1}{\log^pn}
            n^\alpha\bigg(
               1 - \frac{\alpha\log n}{n} + \frac{\alpha\Psi(\alpha)}{n} +O\bigg(\frac{1}{n\log n}\bigg)
            \bigg)
            + \frac{p}{\log^{p+1}n}\big(n^{\alpha-1}\log n+O(n^{\alpha-1})\big)\\
      & = & \frac{n^\alpha}{\log^pn}\bigg(1 - \frac{\alpha\log n}{n}
            + \frac{\alpha\Psi(\alpha)+p}{n} + O\bigg(\frac{1}{n\log n}\bigg)\bigg).
   \end{eqnarray*}
   Thus, it suffices to verify that
   \begin{equation} \label{suffices}
   \sum_{m=2}^{n-1}\frac{m^\alpha}{(n-m)(n-m+1)}\bigg(
      \frac{1}{\log^pm}
      - \frac{1}{\log^pn}
      - p\frac{-\log(m/n)}{\log^{p+1}n}
      \bigg)
      \ =\ O\bigg(\frac{n^{\alpha-1}}{\log^{p+1}n}\bigg).
   \end{equation}
   The function $f_{np}:(1,n]\to\mr$, defined via
   $$
   f_{np}(x)\ :=\ \frac{1}{\log^px} - \frac{1}{\log^pn} - p
   \frac{-\log(x/n)}{\log^{p+1}n},
   $$
   has derivative
   $$
   f_{np}'(x)\ =\ \frac{p}{x}\bigg(
      \frac{1}{\log^{p+1}n}-\frac{1}{\log^{p+1}x}
   \bigg)
   \ \le\ 0
   $$
   and satisfies $f_{np}(n)=0$. Thus, $f_{np}\ge 0$. In order to verify
   (\ref{suffices}) we use a decomposition method. We split up
   the sum on the left hand side in (\ref{suffices}) into two parts
   $\sum_{m=2}^{a_n}\ldots$ and $\sum_{m=a_n+1}^{n-1}\ldots$, and
   handle these two parts separately. We work with the sequence
   $(a_n)_{n\in\mn}$ defined via $a_1:=1$ and
   $a_n:=\lfloor n/\log^{p+1}n\rfloor$ for $n\ge 2$. For the first
   part we obtain
   \begin{eqnarray*}
      0
      & \le & \sum_{m=2}^{a_n}\frac{m^\alpha}{(n-m)(n-m+1)}f_{np}(m)
      \ \le \ \sum_{m=2}^{a_n}\frac{m^\alpha}{(n-m)(n-m+1)}\frac{1}{\log^pm}\\
      & \le &
      \frac{n^\alpha}{\log^p2}\sum_{m=2}^{a_n}\bigg(\frac{1}{n-m}-\frac{1}{n-m+1}\bigg)
      \ =\ O\bigg(\frac{n^{\alpha-1}}{\log^{p+1}n}\bigg),
   \end{eqnarray*}
   since
   \begin{eqnarray*}
      \sum_{m=2}^{a_n}\bigg(\frac{1}{n-m}-\frac{1}{n-m+1}\bigg)
      & = & \frac{1}{n-a_n} - \frac{1}{n-1}
      \ =\ \frac{a_n-1}{(n-a_n)(n-1)}
      \ \sim\ \frac{1}{n\log^{p+1}n}.
   \end{eqnarray*}
   Moreover, for the second part we have
   \begin{eqnarray*}
      &   & \hspace{-1cm}
            \sum_{m=a_n+1}^{n-1}\frac{m^\alpha}{(n-m)(n-m+1)}f_{np}(m)\\
      & = & \sum_{m=a_n+1}^{n-1}
            \frac{m^\alpha}{(n-m)(n-m+1)}
            \Big(\frac{1}{\log^pm}-\frac{1}{\log^pn}-p\frac{-\log(m/n)}{\log^{p+1}n}\Big)\\
      & = & \sum_{m=a_n+1}^{n-1}\frac{m^\alpha}{(n-m)(n-m+1)}
            \frac{\log^pn-\log^pm-p(-\log(m/n))\log^pm/\log n}{\log^pn\log^pm}\\
      & \le & \frac{n^\alpha}{\log^pn\log^p a_n}
            \sum_{m=a_n+1}^{n-1}\frac{\log^pn-\log^pm+p\log(m/n)\log^pm/\log n}{(n-m)(n-m+1)}\\
      & \sim & \frac{n^\alpha}{\log^{2p}n}
            \sum_{m=a_n+1}^{n-1}
            \frac{\log^pn -\log^pm+p\log(m/n)\log^pm/\log n}{(n-m)(n-m+1)},
   \end{eqnarray*}
   since $\log a_n\sim\log n$ as $n\to\infty$. Thus, it remains to
   verify that
   \begin{equation} \label{local3}
   \sum_{m=a_n+1}^{n-1}\frac{\log^pn-\log^pm+p\log(m/n)\log^pm/\log n}{(n-m)(n-m+1)}
   \ =\ O\Big(\frac{\log^{p-1}n}{n}\Big).
   \end{equation}
   Let us distinguish the two cases $p\ge 1$ and $p<1$.
   Suppose first that $p\ge 1$.
   Then the map $x\mapsto x^p$ is convex on $[0,\infty)$. Thus,
   $y^p-x^p\le p(y-x)y^{p-1}$
   for all $x,y\in [0,\infty)$ with $x\le y$. It follows that
   $y^p-x^p+p(x-y)x^p/y\le p^2(y-x)^2y^{p-2}$ for all $x,y\in [0,\infty)$
   with $x<y$. Applying this inequality with $0\le x:=\log m<y:=\log n$
   yields
   \begin{eqnarray*}
      0
      & \le & \sum_{m=a_n+1}^{n-1}
            \frac{\log^pn-\log^pm+p\log(m/n)\log^pm/\log n}{(n-m)(n-m+1)}\\
      & \le & \sum_{m=1}^{n-1}
            \frac{\log^pn-\log^pm+p\log(m/n)\log^pm/\log n}{(n-m)(n-m+1)}\\
      & \le & \sum_{m=1}^{n-1}\frac{p^2(\log n-\log m)^2\log^{p-2}n}{(n-m)(n-m+1)}
      \ = \ p^2\log^{p-2}n\sum_{m=1}^{n-1}\frac{\log^2(m/n)}{(n-m)(n-m+1)}.
   \end{eqnarray*}
Note that
   \begin{eqnarray*}
      n\sum_{m=1}^{n-1}\frac{\log^2(m/n)}{(n-m)(n-m+1)}
      & = & \frac{1}{n}\sum_{m=1}^{n-1}\frac{\log^2(m/n)}{(1-m/n)(1-(m-1)/n)}\\
      & \to & \int_0^1\frac{\log^2x}{(1-x)^2}\,{\rm d}x
      \ =\ \Gamma(3)\zeta(2)
      \ =\ \frac{\pi^2}{3}\ \in\ \mr,
   \end{eqnarray*}
   where the particular value $\Gamma(3)\zeta(2)$ of last integral is
   obtained by choosing $s:=2$ in the chain of equalities
   $$
   \int_0^1 \frac{(-\log(1-u))^s}{u^2}\,{\rm d}u
   \ =\ \int_0^\infty \frac{x^se^x}{(e^x-1)^2}\,{\rm d}x
   \ =\ \int_0^\infty \frac{sx^{s-1}}{e^x-1}\,{\rm d}x
   \ =\ \Gamma(s+1)\zeta(s),\quad s>1,
   $$
   which are based on the substitution $x=-\log(1-u)$,
   partial integration, and on formula 23.2.7 in \cite{Abramovitz_Stegun}.
   Thus, the expression on the left hand side in (\ref{local3})
   is even $O((\log^{p-2}n)/n)$. In particular, (\ref{local3}) holds.
   Suppose now that $p\in [0,1)$. Then the map $x\mapsto x^p$ is
   concave on $[0,\infty)$. Thus, $y^p-x^p\le p(y-x)x^{p-1}$
   for all $x,y\in (0,\infty)$ with $x\le y$. It follows that
   $y^p-x^p+p(x-y)x^p/y\le p(y-x)^2x^{p-1}/y$ for all $x,y\in (0,\infty)$
   with $x\le y$. Applying this inequality with $0<x:=\log m<y:=\log n$
   yields
   \begin{eqnarray*}
      0
      & \le & \sum_{m=a_n+1}^{n-1}\frac{\log^pn-\log^pm+p\log(m/n)\log^pm/\log n}{(n-m)(n-m+1)}\\
      & \le & \sum_{m=a_n+1}^{n-1}\frac{p(\log n-\log m)^2(\log^{p-1}m)/\log n}{(n-m)(n-m+1)}\\
      & \le & \frac{p\log^{p-1}a_n}{\log n}\sum_{m=a_n+1}^{n-1}
              \frac{\log^2(m/n)}{(n-m)(n-m+1)}\\
      & \le & \frac{p\log^{p-1}a_n}{\log n}\sum_{m=1}^{n-1}\frac{\log^2(m/n)}{(n-m)(n-m+1)}
      \ =\ O\bigg(\frac{\log^{p-2}n}{n}\bigg),
   \end{eqnarray*}
   since $\log a_n\sim \log n$ and the last sum is $O(1/n)$ as shown
   above. Again, (\ref{local3}) holds.
\end{proof}
The following Corollary \ref{appmain} is essentially obtained by
combining the three Lemmata \ref{totalrates}, \ref{applemmasimple}
and \ref{applemma}.
It provides an asymptotic expansion for the sum
$\sum_{m=2}^{n-1}p^{(1)}_{n,m}m^\alpha/\log^pm$. This expansion is a key tool for the proof of Theorem \ref{main1}.
\begin{cor} \label{appmain}
   Fix $\alpha\in [1,\infty)$ and $p\in [0,\infty)$. For the
   $\be\,(1,b)$-coalescent with parameter $b\in (0,\infty)$,
   \begin{equation} \label{appmainexpansion}
   \sum_{m=2}^{n-1} p^{(1)}_{n,m}\frac{m^\alpha}{\log^pm}
   \ =\ \frac{n^\alpha}{\log^pn}
   \bigg(
      1-\alpha\frac{\log n}{n} + \frac{c_{b,\alpha,p}}{n}
      + O\bigg(\frac{1}{n\log n}\bigg)
   \bigg),\qquad n\to\infty,
   \end{equation}
   where $c_{b,\alpha,p}:=(\alpha+b-1)\Psi(\alpha+b-1)+p+1+(1-b)\Psi(b)
   =(\alpha+b-1)\Psi(\alpha+b)+p-(b-1)\Psi(b)$.
\end{cor}
\begin{rem}
   The following proof shows that Corollary \ref{appmain} even holds
   for the slightly larger range of parameters $\alpha,b\in (0,\infty)$
   satisfying $\alpha+b-1>0$. However, we need Corollary
   \ref{appmain} only for $\alpha\in [1,\infty)$ and $b\in (0,\infty)$,
   in which case $\alpha+b-1>0$ automatically holds.
\end{rem}
\begin{proof}
   Let $g_{nm}:=\lambda_n\mmp\{I_n=n-m\}$ denote the rate at which
   the block counting process moves from the state $n$ to the state
   $m\in\{1,\ldots,n-1\}$. It suffices to verify that
   \begin{eqnarray}
      &   & \hspace{-12mm}
            \sum_{m=2}^{n-1} g_{nm}\frac{m^\alpha}{\log^pm}\nonumber\\
      & = & b\frac{n^{\alpha+1}}{\log^pn}
           \bigg(1 - (\alpha+b-1)\frac{\log n}{n}
           +\frac{(\alpha+b-1)\Psi(\alpha+b-1)+p}{n}
           +O\bigg(\frac{1}{n\log n}\bigg)\bigg),
           \label{expansion1}
   \end{eqnarray}
   since (\ref{appmainexpansion}) then follows from $p^{(1)}_{n,m}=g_{nm}/\lambda_n$
   by multiplying (\ref{expansion1}) with (\ref{gninverseasy}). Note
   that 
   $$
   g_{nm}
   \ =\ b\frac{n!}{\Gamma(b+n-1)}\frac{1}{(n-m)(n-m+1)}
   \frac{\Gamma(b+m-1)}{(m-1)!},\quad 1\le m<n.
   $$
   Since the first fraction has expansion
   \begin{equation} \label{localexpansion}
      \frac{n!}{\Gamma(b+n-1)}
      \ =\ \frac{1}{n^{b-2}}\bigg(
              1
              - \binom{b-1}{2}\frac{1}{n}
              + O\bigg(\frac{1}{n^2}\bigg)
           \bigg),
   \end{equation}
   it hence suffices to verify that
   \begin{eqnarray} \label{expansion2}
      &   & \hspace{-1cm}
            \sum_{m=2}^{n-1}
            \frac{1}{(n-m)(n-m+1)}
            \frac{\Gamma(b+m-1)}{(m-1)!}\frac{m^\alpha}{\log^pm}\nonumber\\
      & = & \frac{n^{\alpha+b-1}}{\log^pn}
            \bigg(
                1 - (\alpha+b-1)\frac{\log n}{n}\nonumber\\
      &   & \hspace{2cm}
                + \frac{\binom{b-1}{2}+(\alpha+b-1)\Psi(\alpha+b-1)+p}{n}
                + O\bigg(\frac{1}{n\log n}\bigg)
            \bigg),
   \end{eqnarray}
   since (\ref{expansion1}) then follows by multiplying (\ref{expansion2})
   with (\ref{localexpansion}). Thus, it remains to verify
   (\ref{expansion2}). Since for all $m\in\mn$ and all $b\in (0,\infty)$,
   the Pochhammer like expression $\Gamma(b+m-1)/(m-1)!$ appearing
   on the left hand side in (\ref{expansion2}) is bounded below and
   above by
   $$
   m^{b-1}+\binom{b-1}{2}m^{b-2}
   \ \le\ \frac{\Gamma(b+m-1)}{(m-1)!}
   \ \le\ m^{b-1} + \binom{b-1}{2}m^{b-2} + K_b m^{b-3},
   $$
   where $K_b:=\Gamma(b)-1-\binom{b-1}{2}$, (\ref{expansion2})
   follows by plugging in these lower and upper bounds on the left
   hand side in (\ref{expansion2}) and
   applying afterwards Lemma \ref{applemma} with
   $\alpha$ replaced by $\alpha+b-1>0$ and noting that
   $$
   \sum_{m=2}^{n-1}\frac{m^{\alpha+b-2}}{(n-m)(n-m+1)\log^pm}
   \ =\ \frac{n^{\alpha+b-2}}{\log^pn}\bigg(1 + O\bigg(\frac{\log n}{n}\bigg)\bigg)
   $$
   and that
   $$
   \sum_{m=2}^{n-1}\frac{m^{\alpha+b-3}}{(n-m)(n-m+1)\log^pm}
   \ =\ O\bigg(\frac{n^{\alpha+b-3}}{\log^pn}\bigg)
   $$
   by Lemma \ref{applemmasimple}. The proof is complete.
\end{proof}

\end{document}